\documentclass{elsarticle}

\usepackage{amscd}
\usepackage{amsmath}
\usepackage{amssymb}
\usepackage{amsthm}
\usepackage{graphicx}
\usepackage[all]{xy}
\usepackage{kotex}
\usepackage{color}
\usepackage{enumerate}
\theoremstyle{plain}

\usepackage{nameref}

\newtheorem{Thm}{Theorem}[section]
\newtheorem{Lemma}[Thm]{Lemma}

\newtheorem{Cor}[Thm]{Corollary}
\newtheorem{definition}[Thm]{Definition}
\newtheorem*{Thm*}{Theorem}
\theoremstyle{remark}
\newtheorem{Rmk}[Thm]{\bf{Remark}}

\usepackage{color}							%% 색깔
\journal{....}
\begin{document}
	\begin{frontmatter}
		\title{A Condition for Blow-up solutions to Discrete $p$-Laplacian Parabolic Equations under the mixed boundary conditions on Networks\\}
		
		\author[a,b]{Soon-Yeong Chung}
		\ead{sychung@sogang.ac.kr}
		\author[b]{Min-Jun Choi}
		\ead{dudrka2000@sogang.ac.kr}
		\author[b]{Jaeho Hwang\corref{cjhp}}
		\ead{hjaeho@sogang.ac.kr}

		\cortext[cjhp]{Corresponding author}
		\address[a]{National Institute for Mathematical Sciences, Daejeon 34047, Republic of Korea}
		\address[b]{Department of Mathematics, Sogang University, Seoul 04107, Republic of Korea}

		\begin{abstract}
			The purpose of this paper is to investigate a condition
			\begin{center}
				$(C_{p})$ $\hspace{1cm} \alpha \int_{0}^{u}f(s)ds \leq uf(u)+\beta u^{p}+\gamma,\,\,u>0$
			\end{center}
			for some $\alpha>2$, $\gamma>0$, and $0\leq\beta\leq\frac{\left(\alpha-p\right)\lambda_{p,0}}{p}$, where $p>1$ and $\lambda_{p,0}$ is the first eigenvalue of the discrete $p$-Laplacian $\Delta_{p,\omega}$. Using the above condition, we obtain blow-up solutions to discrete $p$-Laplacian parabolic equations
			\begin{equation*}
			\begin{cases}
			u_{t}\left(x,t\right)=\Delta_{p,\omega}u\left(x,t\right)+f(u(x,t)), & \left(x,t\right)\in S\times\left(0,+\infty\right),\\
			\mu(z)\frac{\partial u}{\partial_{p} n}(x,t)+\sigma(z)|u(x,t)|^{p-2}u(x,t)=0, & \left(x,t\right)\in\partial S\times\left[0,+\infty\right),\\
			u\left(x,0\right)=u_{0}\geq0(nontrivial), & x\in S,
			\end{cases}
			\end{equation*}
			on a discrete network $S$, where $\frac{\partial u}{\partial_{p}n}$ denotes the discrete $p$-normal derivative. Here, $\mu$ and $\sigma$ are nonnegative functions on the boundary $\partial S$ of $S$, with $\mu(z)+\sigma(z)>0$, $z\in \partial S$. In fact, it will be seen that the condition $(C_{p})$, the generalized version of the condition $(C)$, improves the conditions known so far.
		\end{abstract}
		
	\end{frontmatter}

%%%%%%%%%%%%%%%%%%%%%%%%%%%%%%%%%%%%%%%%%%%%%%%%%%
\section{Introduction}

These days, the discrete version of differential equations has attracted many researcher's attention.  In particular, $p$-Laplacian $\Delta_{p,\omega}$ on networks(or weighted graphs) is used to observe various social and scientific phenomena(see \cite{ELB}-\cite{CG} and references therein), which is modeled by discrete $p$-Laplacian parabolic equations
\begin{equation*}
u_{t}\left(x,t\right)=\Delta_{p,\omega}+u^{q}\left(x,t\right),\,\, (x,t)\in S\times(0,\infty)
\end{equation*}
with some boundary and initial conditions where $S$ is the set of chemicals and $p>1$. Here, $\Delta_{p,\omega}$ is the discrete $p$-Laplace operator on $S$, defined by
\begin{equation*}
{\displaystyle \Delta_{p,\omega}u(x,t):= \sum_{x\in\overline{S}}\left| u\left(y,t\right)-u\left(x,t\right) \right|^{p-2}\left[u\left(y,t\right)-u\left(x,t\right)\right]\omega\left(x,y\right)}.
\end{equation*}

From a similar point of view, we discuss, in this paper, the blow-up property of solutions to the following discrete $p$-Laplacian parabolic equations
\begin{equation}\label{equation}
\begin{aligned}
\begin{cases}
	u_{t}(x,t) = \Delta_{p,\omega}u(x,t) + f(u(x,t)), & (x,t)\in S\times (0, +\infty), \\
	B[u]=0, & \text{on}\,\,\partial S\times [0,+\infty), \\
	u(x, 0)=u_{0}(x)\geq 0, & x\in \overline{S},
\end{cases}
\end{aligned}
\end{equation}
where $p>1$, $f$ is locally Lipschitz continuous on $\mathbb{R}$, and $B[u]=0$ on $\partial S\times [0,+\infty)$ stands for the boundary condition
\begin{equation}\label{boundary condition}
	\mu(z)\frac{\partial u}{\partial_{p} n}(z,t)+\sigma(z)|u(z,t)|^{p-2}u(z,t)=0,\,\,\left(z,t\right)\in\partial S\times\left[0,+\infty\right).
\end{equation}
Here, $\mu,\sigma:\partial S \rightarrow [0,+\infty)$ are functions with $\mu(z)+\sigma(z)>0$, $z\in \partial S$ and $\frac{\partial u}{\partial_{p} n}$ denotes the discrete $p$-normal derivative (which is introduced in Section 1). It is easy to see that this boundary value problem includes the various
boundary value problems such as the Dirichlet boundary, Neumann boundary, Robin boundary, and so on. We note here that one of the meaning of our result is an unified approach.

The continuous case of this equation with some boundary conditions has been studied by many authors. For example, in $1973$, Levine \cite{L} considered the formally parabolic equations of the form
\[
\begin{cases}
P\frac{du}{dt}=-A(t)u + f(u(t)), & t\in [0, +\infty),\\
u(0)=u_{0},
\end{cases}
\]
where $P$ and $A(t)$ are positive linear operators defined on a dense subdomain $D$ of a real or complex Hilbert space $H$. Here, he first introduced "the concavity method" to obtained the blow-up solutions, under abstract conditions
\[
\begin{aligned}
2(\alpha+1)F(x)\leq (x,f(x)),\,\,F(u_{0}(x))>\frac{1}{2}(u_{0}(X),Au_{0}(x))
\end{aligned}
\]
for every $x\in D$, where $F(x)=\int_{0}^{1}(f(\rho x),x)d\rho$.

After this, Philippin and Proytcheva \cite{PP} have applied the above method to the equations
\begin{equation}\label{IBBequatio}
\begin{cases}
u_{t}=\Delta u + f(u), & \hbox{in}\,\,\Omega\times(0,+\infty),\\
u(x,t)=0, & \hbox{on}\,\, \partial \Omega\times(0,+\infty),\\
u(x,0)=u_{0}(x)\geq0,
\end{cases}
\end{equation}
and obtained a blow-up solution, under the condition
\[
(A)\,:\hspace{2mm} (2+\epsilon)F(u)\leq uf(u),\,\,u>0
\]
and the initial data $u_{0}$ satisfying
\[
-\frac{1}{2}\int_{\Omega}|\nabla u_{0}(x)|^{2}dx + \int_{\Omega}F(u_{0}(x))dx>0.
\]
Besides, in \cite{PPP, PS} Payne et al. obtained the blow-up solutions to the equations
\begin{equation}\label{IBBequation}
\begin{cases}
u_{t}=\Delta u - g(u), & \hbox{in}\,\,\Omega\times(0,+\infty),\\
\frac{\partial u}{\partial n}=f(u), & \hbox{on}\,\, \partial \Omega\times(0,+\infty),\\
u(x,0)=u_{0}(x)\geq0,
\end{cases}
\end{equation}
when the Neumann boundary data $f$ satisfies the condition $(A)$.

Recently, Ding and Hu \cite{DH} adopted the condition $(A)$ to get blow-up solutions to the equation
\[
(g(u))_{t}=\nabla\cdot(\rho(|\nabla u|^{2})\nabla u) + k(t)f(u)
\]
with the nonnegative initial value and the null Drichlet boundary condition.

On the other hands, the condition (A) was relaxed by Bandle and Brunner \cite{BB} as follows:
\[
\begin{aligned}
\hbox{(B) \hspace{2mm} $(2+\epsilon)F(u)\leq uf(u) + \gamma ,\,\,u>0$}
\end{aligned}
\]
and the initial data $u_{0}$ satisfying
\[
-\frac{1}{2}\int_{\Omega}|\nabla u_{0}(x)|^{2}dx+\int_{\Omega}[F(x,u_{0})-\gamma]dx>0,
\]
for some $\epsilon>0$ and $\gamma>0$. \\

Finally, the condition $(B)$ was developed by Chung and Choi \cite{CC2} as follows:
\[
\begin{aligned}
\hbox{(C) \hspace{2mm} $(2+\epsilon)F(u)\leq uf(u)+\beta u^{2} + \gamma ,\,\,u>0$}
\end{aligned}
\]
and the initial data $u_{0}$ satisfying
\[
-\frac{1}{2}\sum_{x\in\overline{S}}|\nabla u_{0}(x)|^{2}+\sum_{x\in\overline{S}}[F(x,u_{0})-\gamma]dx>0,
\]
for some $\epsilon>0$, $0< \beta\leq \frac{\epsilon \lambda_{0}}{2}$, and $\gamma>0$. Here, $\lambda_{0}$ denotes the first eigenvalue of the discrete Laplace operator $\Delta_{\omega}$.

It is easy to see that the conditions $(A)$ and $(B)$ above are independent of the eigenvalue of Laplace operator which depends on the domain and the condition $(C)$ is depend on the eigenvalue.

From this point of view, we generalized the condition $(C)$ with respect to discrete $p$-Laplace operator $\Delta_{p,\omega}$, which is the main results of this paper, will be introduced as follows: for some $\alpha>2$, $\beta\geq 0$, and $\gamma>0$,
\[
\hbox{$(C_{p})$\hspace{3mm} $\alpha F\left(u\right)\leq uf(u) + \beta u^{p} + \gamma,\,\,u>0$},
\]
where $0\leq\beta\leq\frac{\left(\alpha-p\right)\lambda_{p,0}}{p}$, $p>1$, and $\lambda_{p,0}$ is the first eigenvalue of the discrete $p$-Laplacian $\Delta_{p,\omega}$. Here, we note that the term $\beta u^{p}$ is depending on the domain graph.

From this observation, we may understand the condition $(A)$ and $(B)$ with respect to the $p$-Laplace operator $\nabla(|\nabla u|^{p-2}\nabla u)$ as follows: for $p>1$,
\begin{equation*}
\begin{aligned}
&(A_{p})\hspace{3mm} \alpha F(u)\leq uf(u),\,\hspace{1mm}\,u>0,\\
&(B_{p}) \hspace{3mm}\alpha F(u)\leq uf(u)+\gamma,\hspace{1mm}\,\,u>0,
\end{aligned}
\end{equation*}
for some $\alpha>p$ with $\alpha>2$ and $\gamma>0$. Above conditions $(A_{p})$, $(B_{p})$, and $(C_{p})$ are discussed in Section $3$.

As far as the authors know, it seems that there have been no paper which deal with the blow-up solutions to the equation \eqref{equation} for $1<p<2$ in the discrete case, not even in the continuous case.

In fact, it is expected that, with the condition $(C_{p})$, more interesting results should be obtained even in the continuous case, which will be our forth-coming work.

We organize this paper as follows: in Section \ref{preli}, we introduce briefly the preliminary concepts on networks and comparison principles. Section \ref{BCM} is the main section, which is devoted to blow-up solutions using the concavity method with the condition $(C_{p})$. Finally in Section 3, we discuss the condition $(C_{p})$, comparing with the conditions $(A_{p})$ and $(B_{p})$, together with the condition   $B(0)>0$ for the initial data.

%%%%%%%%%%%%%%%%%%%%%%%%%%%%%%%%%%%%%%%%%%%%%%%%%%%%%%%%%%%%%%%%%%%%%%%%%%%%%%%%%%%%%%%%%%%%%%%%%%%%%%%%%%%%%%%%%%%%%%%%%%%%%%%%%%%%%%%%%%%%%%%%%%%%%%

%%%%%%%%%%%%%%%%%%%%%%%%%%%%%%%%%%%%%%%%%%%%%%%%%%%%%%%%%%%%%%%%%%%%%%%%%%%%%%%%%%%%%%%%
\section{Preliminaries and Discrete Comparison Principles}\label{preli}
In this section, we start with the theoretic graph notions frequently used throughout this paper. For more detailed information on notations, notions, and conventions, we refer the reader to \cite{CB}.

\begin{definition}
\begin{enumerate}[(i)]
\item  A graph $G=G\left(V,E\right)$ is a finite set $V$ of $vertices$ with a set $E$ of $edges$ (two-element subsets of $V$). Conventionally used, we denote by $x\in V$ or $x\in G$ the fact that $x$ is a vertex in $G$.
\item A graph $G$ is called $simple$ if it has neither multiple edges nor loops
\item $G$ is called $connected$ if for every pair of vertices $x$ and $y$, there exists a sequence(called a $path$) of vertices $x=x_{0},x_{1},\cdots,x_{n-1},x_{n}=y$ such that $x_{j-1}$ and $x_{j}$ are connected by an edge(called $adjacent$) for $j=1,\cdots,n$.
\item A graph $G'=G'\left(V',E'\right)$ is called a $subgraph$ of $G\left(V,E\right)$ if $V'\subset V$ and $E'\subset E$. In this case, $G$ is a host graph of $G'$. If $E'$ consists of all the edges from $E$ which connect the vertices of $V'$ in its host graph $G$, then $G'$ is called an induced subgraph.
\end{enumerate}
\end{definition}
 We note that an induced subgraph of a connected host graph may not be connected.

Throughout this paper, all the subgraphs are assumed to be induced, simple and connected.

\begin{definition}
	For an induced subgraph $S$ of a graph $G=G\left(V,E\right)$, the (vertex) $boundary$ $\partial S$ of $S$ is defined by
	\[
	\partial S:=\{z\in V\setminus S\,|\, \hbox{$z\sim y$ for some $y\in S$}\}.
	\]
\end{definition}
Also, we denote by $\overline{S}$ a graph whose vertices and edges are in $S\cup\partial S$. We note that by definition the set, $\overline{S}$ is an induced subgraph of $G$.

\begin{definition}
A $weight$ on a graph $G$ is a symmetric function $\omega\,:\,V\times V\rightarrow\left[0,+\infty\right)$ satisfying the following:
\begin{enumerate}[(i)]
\item $\omega\left(x,x\right)=0$, \,\,$x\in V$,
\item $\omega\left(x,y\right)=\omega\left(y,x\right)$ if $x\sim y$,
\item $\omega\left(x,y\right)>0$ if and only if $\{x,y\}\in E$,
\end{enumerate}
and a graph $G$ with a weight $\omega$ is called a $network$.
\end{definition}

\begin{definition}
The \emph{degree} $d_{\omega}x$ of a vertex $x$ in a network $S$ (with boundary $\partial S$) is defined by
\[
d_{\omega}x:=\sum_{y\in\overline{S}}\omega\left(x,y\right).
\]
\end{definition}

\begin{definition}
For $p>1$ and a function $u\,:\,\overline{S}\rightarrow\mathbb{R}$, the discrete $p$-Laplacian $\Delta_{p,\omega}$ on $S$ is defined by
\[
\Delta_{p,\omega}u\left(x\right):=\sum_{y\in\overline{S}}\left\vert u\left(y\right)-u\left(x\right)\right\vert^{p-2}\left[u\left(y\right)-u\left(x\right)\right]\omega\left(x,y\right)
\]
for $x\in S$.
\end{definition}

\begin{definition}
	For $p>1$ and a function $u\,:\,\overline{S}\rightarrow\mathbb{R}$, the discrete $p$-normal derivative $\frac{\partial u}{\partial_{p} n}$ on $\partial S$ is defined by
	\[
	\frac{\partial u}{\partial_{p} n}(z):=\sum_{x\in S}\left\vert u(z)-u(x)\right\vert ^{p-2}[u(z)-u(x)]\omega(x,z)
	\]
	for $z\in \partial S$.
\end{definition}

The following two lemmas are used throughout this paper.
\begin{Lemma}[See \cite{KC}]\label{laplace}
Let $p>1$. For functions $f,~g\,:\,\overline{S}\to\mathbb{R}$,
the discrete $p$-Laplacian $\Delta_{p,\omega}$ satisfies that
\begin{equation*}
\begin{aligned}
&2\sum_{x\in\overline{S}}g\left(x\right)\left[-\Delta_{p,\omega}f\left(x\right)\right]
\\&=\sum_{x,y\in\overline{S}}\left\vert f\left(y\right)-f\left(x\right)\right\vert^{p-2}\left[f\left(y\right)-f\left(x\right)\right]\left[g\left(y\right)-g\left(x\right)\right]\omega\left(x,y\right).	
\end{aligned}
\end{equation*}

In particular, in the case $g=f$, we have
$$
2\sum_{x\in\overline{S}}f\left(x\right)\left[-\Delta_{p,\omega}f\left(x\right)\right]=\sum_{x,y\in\overline{S}}\left\vert f\left(x\right)-f\left(y\right)\right\vert^{p}\omega\left(x,y\right).
$$
\end{Lemma}
\begin{Lemma}[See \cite{CH}]\label{eigenvalue}

For $p>1$, there exist $\lambda_{p,0}>0$ and a function $\phi_{0}\left(x\right)>0$, $x\in S\cup\Gamma$
such that
\[
\begin{cases}
-\Delta_{p,\omega}\phi_{0}\left(x\right)=\lambda_{p,0}|\phi_{0}(x)|^{p-2}\phi_{0}\left(x\right), & x\in S,\\
B[\phi_{0}]=0, & \text{on} \hspace{2mm}\partial S,
\end{cases}
\]
where $B[\phi_{0}]$ on $\partial S$ stands for
\begin{equation*}
	\mu(z)\frac{\partial \phi_{0} }{\partial_{p} n}(z)+\sigma(z)|\phi_{0}(z)|^{p-2}\phi_{0}(z),\,\,z\in \partial S.
\end{equation*}
Here, $\Gamma:=\{z\in \partial S\,|\,\mu(z)>0 \}$ and $\mu,\sigma:\partial S\rightarrow[0,+\infty)$ are functions with $\mu(z)+\sigma(z)>0$ for all $z\in \partial S$. Moreover, $\lambda_{p,0}$ is given by
\[
\begin{aligned}
\lambda_{p,0} =& \min_{u\in\mathcal{A},u\not\equiv0}\frac{\frac{1}{2}{\displaystyle\sum_{x,y\in\overline{S}}}\left\vert u\left(x\right)-u\left(y\right)\right\vert^{p}\omega\left(x,y\right)+\displaystyle\sum_{z\in\Gamma}\frac{\sigma(z)}{\mu(z)}|u(z)|^{p}}{{\displaystyle\sum_{x\in S}}\left|u\left(x\right)\right|^{p}}\\ \leq& d_{\omega}x,\,\,x\in S,
\end{aligned}
\]
where $\mathcal{A}:=\left\{ u\,:\,\overline{S}\rightarrow\mathbb{R}\,|\, u\not\equiv 0\,\,\text{in}\,\, S,\,\,u=0,\,\,\text{on}\,\,\partial S\setminus \Gamma\right\}$.
\end{Lemma}
In the above, the number $\lambda_{p,0}$ is called the first eigenvalue of $\Delta_{p,\omega}$ on a network $\overline{S}$ with corresponding eigenfunction $\phi_{0}$ (see \cite{C} and \cite{CDS} for the spectral theory of the Laplacian operators). In fact, we note that if $\Gamma$ is empty set, then $\sum_{z\in\Gamma}\frac{\sigma(z)}{\mu(z)}|u(z)|^{p}$ implies $0$.
\begin{Rmk}
	It is clear that the first eigenvalue $\lambda_{p,0}$ is nonnegative. Moreover, we note here that the first eigenvalue $\lambda_{p,0}$ satisfies the following statements:
	\begin{itemize}
		\item[(i)] If $\sigma\equiv 0$, then $\lambda_{p,0}=0$.
		\item[(ii)] If $\sigma\not\equiv 0$, then $\lambda_{p,0}>0$.
	\end{itemize}
\end{Rmk}

%%%%%%%%%%%%%%%%%%%%%%%%%%%%%%%%%%%%%%%%%%%%%%%%%%%%%%%%%%%%%%%%%%%%%%%%%%%%%%%%%%%%%%%%%%%%%%%%%%%%%%%%%%%%%%%%%%%%%%%%%%%%%%%%%%%%%%%%%%%%%%%%%%%%%

We now discuss the local existence of a solution to the equation \eqref{equation} which is
\begin{equation*}
\begin{cases}
u_{t}\left(x,t\right)=\Delta_{p,\omega}u\left(x,t\right)+f\left(u\left(x,t\right)\right), & \left(x,t\right)\in S\times\left(0,+\infty\right),\\
B[u]=0, & \text{on} \hspace{2mm}\partial S \times [0,+\infty),\\
u\left(x,0\right)=u_{0}\left(x\right)\geq0, & x\in S,
\end{cases}
\end{equation*}
where $p>1$ and $f$ is locally Lipschitz continuous on $\mathbb{R}$. Here, $B[u]$ on $\partial S\times [0,+\infty)$ stands for the boundary condition \eqref{boundary condition} which is
\begin{equation*}
\mu(z)\frac{\partial u }{\partial_{p} n}(z,t)+\sigma(z)|u(z,t)|^{p-2}u(z,t),\,\,(z,t)\in \partial S\times [0,+\infty),
\end{equation*}
where $\mu,\sigma:\partial S\rightarrow[0,+\infty)$ are functions with $\mu(z)+\sigma(z)>0$ for all $z\in \partial S$. 
\begin{Rmk}\label{compatibility}
	Consider a function $\psi:\mathbb{R}\rightarrow \mathbb{R}$ by
	\begin{equation*}
	\psi(\gamma):=\sum_{x\in S}|\gamma-u(x,t)|^{p-2}\left[\gamma-u(x,t)\right]a(x)+b|\gamma|^{p-2}\gamma,
	\end{equation*}
	where $a(x)\geq 0$ for all $x\in S$, $b\geq 0$ with $a(x)+b>0$ for some $x\in S$. Then it is easy to see that $\psi$ is a continuous function which is strictly increasing and bijective on $\mathbb{R}$. Therefore, there exists $\rho\in\mathbb{R}$ uniquely such that $\psi(\rho)=0$. It means that for all $z\in \partial S$, we can define the value of $u(z,0)$ uniquely according to the boundary condition $B[u]=0$ and initial data $u_{0}$ which are given. i.e. for every $z\in \partial S$, $u(z,0)$ is determined such that
	$$
	\mu(z)\frac{\partial u}{\partial_{p} n}(z,0)+\sigma(z)|u(z,0)|^{p-2}u(z,0)=0,\,\,z\in \partial S,
	$$
	where $\mu,\sigma:\partial S\rightarrow[0,+\infty)$ are given functions with $\mu(z)+\sigma(z)>0$ for all $z\in \partial S$.
\end{Rmk}
\begin{Rmk}
	Considering the initial data with the boundary condition $B[u]=0$ on $\partial S \times [0,+\infty)$, we have compatible condition
	\begin{equation*}
		\mu(z)\frac{\partial u_{0} }{\partial_{p} n}(z)+\sigma(z)|u_{0}(z)|^{p-2}u_{0}(z),\,\,z\in \partial S.
	\end{equation*}
\end{Rmk}
We will use the Schauder fixed point theorem to prove local existence of the equation \eqref{equation}. For this reason, we need the modified version of the Arzel\'a-Ascoli theorem as follows.

\begin{Lemma}[Modified version of the Arzel\'a-Ascoli theorem]\label{Arzela}	
	Let K be a compact subset of $\mathbb{R}$ and $\overline{S}$ be a network. Consider a Banach space $ C\left( \overline{S} \times K \right)$ with the maximum norm $\lVert u \rVert_{\overline{S},K} := \max_{x\in \overline{S}} \max_{t\in K} \left\vert u\left(x,t\right)\right\vert $. Then a subset $A$ of $C\left(\overline{S} \times K\right)$ is relatively compact if A is uniformly bounded on $\overline{S} \times K$ and $A$ is equicontinuous on $K$ for each $x\in \overline{S}$.
\end{Lemma}
\begin{proof}
	The proof of this version is similar to the original one (see \cite{B}). Thus we only state the idea of the proof. Let $\epsilon>0$ be arbitrarily given. Since $K$ is compact on $\mathbb{R}$ and A is equicontinuous on $K$, there is a finite open cover $\left\{N_{1}\left(t_i, \delta_i\right)\right\}$ of $K$ such that
	\begin{center}
		$\left\vert f\left(x,t\right)-f\left(x,t_i\right) \right\vert <\frac{\epsilon}{4} \hspace{1mm} \text{for all} \hspace{1mm} t\in N_{1}\left(t_i,\delta_i\right),\hspace{1mm} i=1,\dots,n ,\hspace{1mm} x \in \overline{S},\hspace{1mm} f \in A$.
	\end{center}
	Define $E=\left\{f\left(x,t_i\right)\vert\, x \in \overline{S} ,\, i=1,\dots,n,\, f \in A\right\} $. Then $E$ is totally bounded, since A is uniformly bounded. Hence there is a sequence $\left\{\xi_{j} \right\}^{m}_{j=1}$ in $\mathbb{R}$ such that
	\begin{center}
		$E \subset \bigcup^{m}_{j=1} N_{2}\left(\xi_{j},\frac{\epsilon}{4}\right)$.
	\end{center}
	Now, set $F:=\left\{k\,:\,\overline{S}\times\{1,\dots,n\}\rightarrow\{1,\dots,m\} \,\vert\, \text{k is a function}\right\}$ and define
	\begin{center}
		$A_k:=\{f\in A\,\vert\, f\left(x,t_i\right) \in N_{2}\left(\xi_{k_{\left(x,i\right)}},\frac{\epsilon}{4}\right),\,\, x\in\overline{S},\,\, i=1,\dots,n\}$ for each $k\in F$.
	\end{center}
	Then we have to show $A \subset \bigcup_{k\in K} A_{k}$. Let $f\in A$ be fixed. For each $x\in\overline{S}$, $i=1,\cdots,n$, $f\left(x,t_i\right)\in E\subset \bigcup_{j=1}^{m}N_{2}\left(\xi_{j},\frac{\epsilon}{4} \right)$. i.e. there is $j={k_{\left(x,i\right)}}$ such that $f\left(x,t_i \right)\in N_{2}\left(\xi_{j},\frac{\epsilon}{4}\right)$. Thus, $A \subset \bigcup_{k\in F} A_{k}$. 
	
	We now claim that the diameter of each $A_k$ is less than $\epsilon$. For each $f,g\in A_{k}$ and $(x,t)\in\overline{S}\times F$, there exists $1\leq i\leq n$ such that $t\in N_{1}\left(t_i,\delta_i\right)$ and
	\[
	\begin{aligned}
	\left\vert f\left(x,t\right)-g\left(x,t\right)\right\vert  \leq & \left\vert f\left(x,t\right)-f\left(x,t_i\right)\right\vert + \left\vert f\left(x,t_i\right)-\xi_{k_{\left(x,i\right)}} \right\vert \\ &  + \left\vert \xi_{k_{\left(x,i\right)}} - g\left(x,t_i\right) \right\vert + \left\vert g\left(x,t_i\right)-g\left(x,t\right)\right\vert < 4 \cdot \frac{\epsilon}{4}=\epsilon.
	\end{aligned}
	\]
	Hence, $A$ is totally bounded and the proof is complete.
\end{proof}
\begin{Thm}[Local existence]
	There exists $t_{0}>0$ such that the equation \eqref{equation} admits at least one bounded solution $u$ such that $u(x,\cdot)$ is continuous on $[0,t_{0}]$ and differentiable in $(0,t_{0})$, for each $x\in \overline{S}$.
\end{Thm}
\begin{proof}
	We first start with the following Banach space:
	\begin{equation*}
	C(S\times [0,t_{0}]):=\left\{ u:S\times [0,t_{0}]\rightarrow \mathbb{R} \,\vert \,u(x,\cdot)\in C([0,t_{0}])\,\,\text{for each}\,\,x\in S \right\}
	\end{equation*}
	with the maximum norm $\lVert u\rVert_{S,t_{0}}:=\max_{x\in S}\max_{0\leq t\leq t_{0}}|u(x,t)|$, where $t_{0}\in \mathbb{R}$ is a positive constant which will be defined later. Now, consider a subspace
	\begin{equation*}
		B_{t_{0}}:=\left\{ u\in C(S\times [0,t_{0}]) \,|\,\lVert u\rVert_{S,t_{0}}\leq 2\lVert u_{0}\rVert_{S,t_{0}} \right\}
	\end{equation*}
	of a Banach space $C(S\times [0,t_{0}])$. Then it is clear that $B_{t_{0}}$ is convex. In order to apply the Schauder fixed point theorem, we have to show that $B_{t_{0}}$ is closed. Let $g_{n}$ be a sequence in $B_{t_{0}}$ which converges to $g$. Since the convergence is uniform, $g$ is continuous. Moreover, $\left|\lVert g_{n} \rVert_{S,t_{0}}- \lVert g \rVert_{S,t_{0}} \right| \leq \lVert g_{n}-g\rVert_{S,t_{0}}$ implies that $g\in B_{t_{0}}$. Hence, $B_{t_{0}}$ is closed.
	
	On the other hand, for every $u\in B_{t_{0}}$, we can define the value of $u(z,t)$ uniquely according to the boundary condition $B[u]=0$ by the similar way to Remark \ref{compatibility}. i.e. for every $u\in B_{t_{0}}$, $u(z,t)$ satisfies
	$$
	\mu(z)\frac{\partial u}{\partial_{p} n}(z,t)+\sigma(z)|u(z,t)|^{p-2}u(z,t)=0,\,\,(z,t)\in \partial S\times[0,t_{0}]
	$$
	for all $(z,t)\in \partial S \times [0,t_{0}]$, where $\mu,\sigma:\partial S\rightarrow[0,+\infty)$ are given functions with $\mu(z)+\sigma(z)>0$ for all $z\in \partial S$.
	Then by the boundary condition, it is clear that $u(z,t)$ satisfies $|u(z,t)|\leq \lVert u\rVert_{S,t_{0}}$, $(z,t)\in \partial S\times [0,t_{0}]$.
	
	Let us define an operator $D:B_{t_{0}} \rightarrow B_{t_{0}}$ by
	\begin{equation*}
	D[u](x,t):=u_{0}(x)+\int_{0}^{t}\Delta_{p,\omega}u(x,s)+f(u(x,s))\,ds,\,\,(x,t)\in S\times [0,t_{0}],
	\end{equation*}
	where $u_{0}:\overline{S}\rightarrow \mathbb{R}$ is a given function.

	Since $f$ is locally Lipschitz continuous on $\mathbb{R}$, there exists $L>0$ such that
	$$|f(a)-f(b)|\leq L|a-b|,\,\,a,b\in[-m,m],$$
	where $m=2\lVert u_{0}\rVert_{S,t_{0}}$. Now, put
	$$t_{0}:=\frac{\lVert u_{0}\rVert_{S,t_{0}}}{\omega_{0}(4\lVert u_{0}\rVert_{S,t_{0}})^{p-1}+4L\lVert u_{0}\rVert_{S,t_{0}}},$$
	where $\omega_{0}:=\max_{x\in\overline{S}}\sum_{y\in\overline{S}}\omega(x,y)$. Then, it is easy to see that the operator $D$ is well-defined. Now, we will show that $D$ is continuous. The verification of the continuity is divided into 2 cases as follows:
	\newline$(i)\,\,1<p<2$
	\newline For $u$ and $v$ in $B_{t_{0}}$, it follows that
	\begin{equation*}
	\begin{aligned}
		\left| D[u](x,t)-D[v](x,t) \right|& \leq\left|\int_{0}^{t_{0}}\sum_{y\in\overline{S}}2^{2-p}\lVert u-v\rVert_{S,t_{0}}^{p-1}\omega(x,y)+L\lVert u-v\rVert_{S,t_{0}} \,ds\right| \\&\leq t_{0}\left[ 2^{2-p}\omega_{0}\lVert u-v\rVert_{S,t_{0}}^{p-1}+L \lVert u-v \rVert_{S,t_{0}} \right].
	\end{aligned}
	\end{equation*}
	\newline$(ii)\,\,p\geq 2$
	\newline For $u$ and $v$ in $B_{t_{0}}$, we have
	\begin{equation*}
	\begin{aligned}
	&\left| D[u](x,t)-D[v](x,t) \right|\\& \leq\left|\int_{0}^{t_{0}}\sum_{y\in\overline{S}}2^{2p-3}(p-1)\lVert u_{0}\rVert_{S,t_{0}}^{p-2}\lVert u-v\rVert_{S,t_{0}}\omega(x,y)+L\lVert u-v\rVert_{S,t_{0}} \,ds\right| \\&\leq t_{0}\left[ 2^{2p-3}(p-1)\lVert u_{0}\rVert_{S,t_{0}}^{p-2}\omega_{0}\lVert u-v\rVert_{S,t_{0}}+L \lVert u-v \rVert_{S,t_{0}} \right].
	\end{aligned}
	\end{equation*}
	Consequently, for each $p>1$, we obtain
	\begin{equation*}
		\lVert D[u]-D[v]\rVert_{S,t_{0}} \leq C_{1}\lVert u-v\rVert_{S,t_{0}}^{p-1}+C_{2}\lVert u-v\rVert_{S,t_{0}}
	\end{equation*}
	where $C_{1}$ and $C_{2}$ are constant depending only on $u_{0}$, $t_{0}$, $p$, $L$ and $\omega_{0}$. Therefore, we obtain the continuity of $D$.
	
	Finally, we will show that $D(B(t_{0}))$ is relatively compact. By Lemma \ref{Arzela}, it is enough to show that $D(B(t_{0}))$ is uniformly bounded on $S\times [0,t_{0}]$ and equicontinuous on $[0,t_{0}]$. Since $D(B(t_{0}))\in B(t_{0})$, it is trivial that $D(B(t_{0}))$ is uniformly bounded. On the other hand, it follows that for each $x\in S$,
	\begin{equation*}
		|D[u](x,t_{1})-D[u](x,t_{2})|\leq |t_{1}-t_{2}|\left[\omega_{0}(4\lVert u_{0}\rVert_{S,t_{0}})^{p-1}+4L\lVert u_{0}\rVert_{S,t_{0}}\right]
	\end{equation*}
	for all $t_{1},t_{2}\in[0,t_{0}]$ and $u\in B_{t_{0}}$, which implies that $D(B(t_{0}))$ is equicontinuous on $[0,t_{0}]$. Hence, $D(B(t_{0}))$ is relatively compact by Lemma \ref{Arzela}. Therefore, there exists $u\in B(t_{0})$ satisfying $D[u]=u$ and boundary condition $B[u]=0$, by the Schauder fixed point theorem. It is clear that $u$ is the solution to the equation \eqref{equation}. On the other hand, it is easy to see that $u$ is bounded. Moreover, $u(x,\cdot)$ is continuous on $ [0,t_{0}]$ and differentiable in $(0,t_{0})$, for each $x\in \overline{S}$, by the definition of $D$ and the boundary condition $B[u]=0$.
\end{proof}
Now, we state two types of comparison principles.
\begin{Thm}[Comparison Principle]\label{comparison principle}
Let $T>0$ ($T$ may be $+\infty$), $p>1$, and $f$ be locally Lipschitz continuous on $\mathbb{R}$. Suppose that real-valued functions $u(x,\cdot)$, $v(x,\cdot) \in C[0, T]$ are differentiable in $(0, T)$ for each $x \in \overline{S}$ and satisfy
\begin{equation}\label{eq1-2}
\begin{cases}
 u_{t}(x,t)-\Delta_{p,\omega}u(x,t)-f(u(x,t)) & \\
\geq v_{t}(x,t)-\Delta_{p,\omega}v(x,t)-f(v(x,t)), & (x,t)\in S\times\left(0,T\right),\\
\hspace{1mm}
 B[u]\geq B[v], & \text{on}\,\,\partial S\times[0,T),\\
 \hspace{1mm}
 u\left(x,0\right)\geq v\left(x,0\right), & x\in\overline{S}.
\end{cases}
\end{equation}
Then $u\left(x,t\right)\geq v\left(x,t\right)$ for all $(x,t)\in\overline{S}\times[0,T).$
\end{Thm}
\begin{proof}
	Let $T'>0$ be arbitrarily given with $T'<T$. Since $f$ is locally Lipschitz continuous on $\mathbb{R}$, there exists $L>0$ such that
	\begin{equation}\label{eq1-3}
	\left|f\left(a\right)-f\left(b\right)\right|\leq L\left|a-b\right|,\,\,a,b\in[-m,m]
	\end{equation}
	where $m={\displaystyle \max_{x\in\overline{S}}\max_{0\leq t\leq T'}}\left\{\left|u\left(x,t\right)\right|,\left|v\left(x,t\right)\right|\right\}.$
	Let $\tilde{u}, \tilde{v}\,:\,\overline{S}\times\left[0,T'\right]\rightarrow\mathbb{R}$ be the functions defined by
	\[
	\tilde{u}(x,t):=e^{-2Lt}u\left(x,t\right),\,\,
	(x,t)\in\overline{S}\times[0,T'].
	\]
	\[
	\tilde{v}(x,t):=e^{-2Lt}v\left(x,t\right),\,\,
	(x,t)\in\overline{S}\times[0,T'].
	\]
	Then from \eqref{eq1-2}, we have
	\begin{equation}\label{eq1-4}
	\begin{aligned}
	&\left[\tilde{u}_{t}\left(x,t\right)-\tilde{v}_{t}\left(x,t\right)\right]-e^{2L(p-2)t}\left[\Delta_{p,\omega}\tilde{u}\left(x,t\right)-\Delta_{p,\omega}\tilde{v}\left(x,t\right)\right]\\ &+2L\left[\tilde{u}\left(x,t\right)-\tilde{v}\left(x,t\right)\right]-e^{-2Lt}\left[f\left(u\left(x,t\right)\right)-f\left(v\left(x,t\right)\right)\right]\geq0
	\end{aligned}
	\end{equation}
	for all $\left(x,t\right)\in S\times(0,T']$.

	We recall that $\tilde{u}(x,\cdot)$ and $\tilde{v}(x,\cdot)$ are continuous on $[0,T']$ for each $x\in \overline{S}$ and $\overline{S}$ is finite. Hence, we can find
	$\left(x_{0},t_{0}\right)\in\overline{S}\times\left[0,T'\right]$ such that
	\[
	\left(\tilde{u}-\tilde{v}\right)\left(x_{0},t_{0}\right)={\displaystyle \min_{x\in\overline{S}}\min_{0\leq t\leq T'}\left(\tilde{u}-\tilde{v}\right)\left(x,t\right)},
	\]
	which implies that
	\begin{equation}\label{123}
	\tilde{v}\left(y,t_{0}\right)-\tilde{v}\left(x_{0},t_{0}\right)\leq \tilde{u}\left(y,t_{0}\right)-\tilde{u}\left(x_{0},t_{0}\right),\,\,\,y\in\overline{S}.
	\end{equation}

	Then now we have only to show that $\left(\tilde{u}-\tilde{v}\right)\left(x_{0},t_{0}\right)\geq0$.

	Suppose that $\left(\tilde{u}-\tilde{v}\right)\left(x_{0},t_{0}\right)<0$, on the contrary. Assume  that $x_{0}\in\partial S$. Then we see that
	\begin{equation}\label{boundary comparison principle}
	\begin{aligned}
	0 \leq&\mu\left(x_{0}\right)\sum_{x\in S}\left[|\tilde{u}\left(x_{0},t_{0}\right)-\tilde{u}\left(x,t_{0}\right)|^{p-2}\left(\tilde{u}\left(x_{0},t_{0}\right)-\tilde{u}\left(x,t_{0}\right)\right)\right.\\&-\left.|\tilde{v}\left(x_{0},t_{0}\right)-\tilde{v}\left(x,t_{0}\right)|^{p-2}\left(\tilde{v}\left(x_{0},t_{0}\right)-\tilde{v}\left(x,t_{0}\right)\right)\right]\omega\left(x_{0},x\right)\\
	&+\sigma\left(x_{0}\right)\left(\tilde{u}\left(x_{0},t_{0}\right)-\tilde{v}\left(x_{0},t_{0}\right)\right)
	\end{aligned}
	\end{equation}
	Therefore, if $\sigma(x_{0})>0$ then the equation \eqref{boundary comparison principle} is negative, which leads a contradiction. If $\sigma(x_{0})=0$, then we have
	\[
	\tilde{u}(x_{0},t_{0})-\tilde{v}(x_{0},t_{0})=\tilde{u}(x,t_{0})- \tilde{v}(x,t_{0})
	\]
	for all $x\in S$. Hence, there exists $x_{1}\in S$ such that 
	\[
	\tilde{u}(x_{0},t_{0})-\tilde{v}(x_{0},t_{0})=\tilde{u}(x_{1},t_{0})- \tilde{v}(x_{1},t_{0}).
	\]
	Hence we may choose $x_{0}\in S$. Moreover, since $\tilde{u}(x,0)-\tilde{v}(x,0)\geq0$ on $\overline{S}$, we have $\left(x_{0},t_{0}\right)\in S\times(0,T']$. Then we obtain from \eqref{123} that
	\begin{equation}\label{eq1-5}
	\Delta_{p,\omega}\tilde{u}\left(x_{0},t_{0}\right)-\Delta_{p,\omega}\tilde{v}\left(x_{0},t_{0}\right)\geq 0
	\end{equation}
	and it follows from the differentiability of $\left(\tilde{u}-\tilde{v}\right)\left(x,t\right)$ in $(0,T']$ for each $x\in\overline{S}$ that
	\begin{equation}\label{eq1-6}
	\left(\tilde{u}_{t}-\tilde{v}_{t}\right)\left(x_{0},t_{0}\right)\leq 0.
	\end{equation}

	According to \eqref{eq1-3}, we have
	\begin{equation}\label{eq1-7}
	\begin{aligned}
	&2L\left[\tilde{u}\left(x_{0},t_{0}\right)-\tilde{v}\left(x_{0},t_{0}\right)\right]-e^{-2Lt_{0}}\left[f\left(u\left(x_{0},t_{0}\right)\right)-f\left(v\left(x_{0},t_{0}\right)\right)\right]\\
	&\leq2L\left[\tilde{u}\left(x_{0},t_{0}\right)-\tilde{v}\left(x_{0},t_{0}\right)\right]+Le^{-2Lt_{0}}\left|u\left(x_{0},t_{0}\right)-v\left(x_{0},t_{0}\right)\right|\\
	&=2L\left[\tilde{u}\left(x_{0},t_{0}\right)-\tilde{v}\left(x_{0},t_{0}\right)\right]+L\left|\tilde{u}\left(x_{0},t_{0}\right)-\tilde{v}\left(x_{0},t_{0}\right)\right|\\
	&=L\left[\tilde{u}\left(x_{0},t_{0}\right)-\tilde{v}\left(x_{0},t_{0}\right)\right]<0,
	\end{aligned}
	\end{equation}
	since $\tilde{u}\left(x_{0},t_{0}\right)<\tilde{v}\left(x_{0},t_{0}\right)$. Combining \eqref{eq1-5}, \eqref{eq1-6}, \eqref{eq1-7}, we obtain the following:
	\[
	\begin{aligned}
	&\tilde{u}\left(x_{0},t_{0}\right)-\tilde{v}\left(x_{0},t_{0}\right)-\left[\Delta_{p,\omega}\tilde{u}\left(x_{0},t_{0}\right)-\Delta_{p,\omega}\tilde{v}\left(x_{0},t_{0}\right)\right]\\
	&+2L\left[\tilde{u}\left(x_{0},t_{0}\right)-\tilde{v}\left(x_{0},t_{0}\right)\right]-e^{-2Lt_{0}}\left[f\left(u\left(x_{0},t_{0}\right)\right)-f\left(v\left(x_{0},t_{0}\right)\right)\right]<0,
	\end{aligned}
	\]
	which contradicts \eqref{eq1-4}. Therefore $\tilde{u}\left(x,t\right)\geq\tilde{v}\left(x,t\right)$ for all $(x,t)\in S\times(0,T']$ so that we get $u\left(x,t\right)\geq v\left(x,t\right)$ for all $(x,t)\in\overline{S}\times[0,T)$, since $T'<T$ is arbitrarily given.
\end{proof}

When $p\geq 2$, we obtain a strong comparison principle as follows:
\begin{Thm}[Strong Comparison Principle]\label{SCP}
Let $T>0$, $(T\, may\,be +\infty)$, $p\geq 2$, and $f$ be locally Lipschitz continuous on $\mathbb{R}$. Suppose that real-valued functions $u(x,\cdot)$, $v(x,\cdot) \in C[0, T]$ are differentiable in $(0, T)$ for each $x \in \overline{S}$ and satisfy
\begin{equation}\label{eq1-8}
\begin{cases}
 u_{t}(x,t)-\Delta_{p,\omega}u(x,t)-f(u(x,t)) & \\
 \geq v_{t}(x,t)-\Delta_{p,\omega}v(x,t)-f(v(x,t)), & (x,t)\in S\times\left(0,T\right),\\
 B[u]\geq B[v], & \text{on}\,\,\partial S\times[0,T),\\
 u\left(x,0\right)\geq v\left(x,0\right), & x\in\overline{S}.
\end{cases}
\end{equation}
If $u\left(x^{*},0\right)>v\left(x^{*},0\right)$ for some $x^{*}\in S$, then $u\left(x,t\right)>v\left(x,t\right)$ for all $(x,t)\in S\cup\Gamma\times(0,T).$
\end{Thm}

\begin{proof}
	First, note that $u\geq v$ on $\overline{S}\times[0,T)$ by above theorem. Let $T'>0$ be arbitrarily given with $T'<T$. Since $f$ is locally Lipschitz continuous on $\mathbb{R}$, there exists $L>0$ such that
	\begin{equation}\label{eq1-9}
	\left|f\left(a\right)-f\left(b\right)\right|\leq L\left|a-b\right|,\,\,a,b\in[-m,m]
	\end{equation}
	where $m={\displaystyle \max_{x\in\overline{S}}\max_{0\leq t\leq T'}}\left\{\left|u\left(x,t\right)\right|,\left|v\left(x,t\right)\right|\right\}.$
	Let $\tau\,:\,\overline{S}\times\left[0,T'\right]\rightarrow\mathbb{R}$ be the functions defined by
	\[
	\tau\left(x,t\right) := u\left(x,t\right)-v\left(x,t\right),\,\,\left(x,t\right)\in\overline{S}\times\left[0,T'\right].
	\]

	Then $\tau\left(x,t\right)\geq0$ for all $\left(x,t\right)\in\overline{S}\times\left[0,T'\right]$. From the inequality \eqref{eq1-8}, we have
	\begin{equation}\label{equation0}
	\tau_{t}(x^{*},t)-\Delta_{p,\omega}u\left(x^{*},t\right)-\Delta_{p,\omega}v(x^{*},t)-\left[f\left(u(x^{*},t)\right)-f\left(v(x^{*},t)\right)\right]\geq0.
	\end{equation}
	for all $0<t\leq T'$. Then by the mean value theorem, for each $y\in \overline{S}$ and $0\leq t\leq T'$, it follows that
	\begin{equation}\label{111}
	\begin{aligned}
			&|u(y,t)-u(x^{*},t)|^{p-2}[u(y,t)-u(x^{*},t)]-|v(y,t)-v(x^{*},t)|^{p-2}[v(y,t)-v(x^{*},t)]\\&=(p-1)|\zeta(x^{*},y,t)|^{p-2}[\tau(y,t)-\tau(x,t)],
	\end{aligned}
	\end{equation}
	where $|\zeta(x^{*},y,t)|\leq2\max_{x\in\overline{S}}\max_{0\leq t\leq T'}{|u(x,t)|,|v(x,t)|}$.	Using \eqref{eq1-9} and \eqref{111}, the inequality \eqref{equation0} becomes
	\begin{equation*}
	\begin{aligned}
	& \tau_{t}\left(x^{*},t\right)\\
	& \geq-d_{\omega}x^{*}(p-1)[2M]^{p-2}\tau\left(x^{*},t\right)-L|\tau\left(x^{*},t\right)|\\
	&=-\left(d_{\omega}x^{*}(p-1)[2M]^{p-2}+L\right)\tau\left(x^{*},t\right).
	\end{aligned}
	\end{equation*}
	This implies
	\begin{equation}\label{equation1}
	\tau\left(x^{*},t\right)\geq\tau\left(x^{*},0\right)e^{-\left(d_{\omega}x^{*}(p-1)[2M]^{p-2}+L\right)t}>0,\,\,t\in(0,T'],
	\end{equation}
	since $\tau\left(x^{*},0\right)>0$. Now, suppose there exists $(x_{0},t_{0})\in S\times (0,T']$ such that
	\begin{center}
		$\tau\left(x_{0},t_{0}\right)=\displaystyle \min_{x\in S\cup\Gamma,\,0<t\leq T'} \tau\left(x,t\right)=0$.
	\end{center}
	Case 1: $x_{0}\in S$.\\
	Since $\tau(x_{0},t_{0})\leq \tau(x,t)$ for all $(x,t)\in \overline{S}\times [0,T']$, We have
	\[
	\tau_{t}\left(x_{0},t_{0}\right)\leq0
	\]
	and
	\[
	\Delta_{p,\omega}u\left(x_{0},t_{0}\right)-\Delta_{p,\omega}v\left(x_{0},t_{0}\right)\geq0.
	\]
	Hence, from the inequality \eqref{eq1-8}, we obtain
	\begin{center}
		$0\leq\tau_{t}\left(x_{0},t_{0}\right)-\Delta_{p,\omega}u\left(x_{0},t_{0}\right)+\Delta_{p,\omega}v\left(x_{0},t_{0}\right)\leq0$.
	\end{center}
	Therefore, we have
	\[
	\Delta_{p,\omega}u\left(x_{0},t_{0}\right)-\Delta_{p,\omega}v\left(x_{0},t_{0}\right)=0,
	\]
	which implies that $\tau\left(y,t_{0}\right)=0$ for all $y\in \overline{S}$ with $y\sim x_{0}$.
	Now, for any $x \in \overline S,$ there exists a path
	\begin{displaymath}
	x_{0} \sim x_{1} \sim \cdots \sim x_{n} \sim x,
	\end{displaymath}
	since $\overline S$ is connected. By applying the same argument as above inductively we see that $\tau(x, t_{0})=0$ for every $x \in \overline S$, which is a contradiction to \eqref{equation1}.\\
	Case 2: $x_{0}\in \Gamma$.\\
	By the boundary condition in \eqref{eq1-8}, we have
	\begin{equation*}
	\begin{aligned}
	&\mu(x)\left[\frac{\partial u}{\partial_{p} n}(x_{0},t_{0})-\frac{\partial u}{\partial_{p} n}(x_{0},t_{0})\right]\\&\geq\sigma(x)[\vert u(x_{0},t_{0})\vert^{p-2}u(x_{0},t_{0})-\vert u(x,t)\vert^{p-2}u(x_{0},t_{0})]=0,
	\end{aligned}
	\end{equation*}
	which follows that
	\begin{equation*}
	\sum_{x\in S}\left[ -|u(x,t_{0})|^{p-2}u(x,t_{0})+|v(x,t_{0})|^{p-2}v(x,t_{0}) \right]\omega(x,x_{0})\geq 0.
	\end{equation*}
	It means that there exists $x_{1}\in S$ with $x_{0}\sim x_{1}$ such that $\tau(x_{1},t_{0})=0$, which contradicts to Case 1. Hence, we finally obtain that $u\left(x,t\right)>v\left(x,t\right)$ for all $\left(x,t\right)\in S\times\left(0,T\right)$, since $T'<T$ is arbitrarily given.
\end{proof}

We note that by the comparison principle, if $f(0)=0$ then solutions $u$ to the equation \eqref{equation} are nonnegative. On the other hand, it is natural that $f$ is assumed to be positive on $(0,+\infty)$ when we deal with the blow-up theory. Hence, we always assume that $f$ is a locally Lipschitz continuous function on $\mathbb{R}$ which is positive in $(0,+\infty)$ and, $f(0)=0$. Moreover, we assume that the initial data $u_{0}$ is nontrivial and nonnegative.

%%%%%%%%%%%%%%%%%%%%%%%%%%%%%%%%%%%%%%%%%%%%%%%%%%%%%%%%%%%%%%%%%%%%%%%%%%%%%%%%%%%%%%%%%%%%%%%%%%%%%%%%%%%%%%%%%%%%%%%%%%%%%%%%%%%%%%%%%%%%%%%%%%%%%%
%%%%%%%%%%%%%%%%%%%%%%%%%%%%%%%%%%%%%%%%%%%%%%%%%%%%%%%%%%%%%%%%%%%%%%%%%%%%%%%%%%%%%%%%%%%%%%%%%%%%%%%%%%%%%%%%%%%%%%%%%%%%%%%%%%%%%%%%%%%%%%%%%%%%%%
\section{Blow-Up: the Concavity Method}\label{BCM}
In this section, we discuss the blow-up phenomena of the solutions to the equation \eqref{equation} by using concavity method, which is the main part of this paper. This method, introduced by Levine \cite{L}, uses the concavity of an auxiliary function. In fact, the concavity method is an elegant tool for deriving estimates and giving criteria for blow-up.

\begin{definition}[Blow-up]
	We say that a solution $u$ to the equation \eqref{equation} blows up at finite time $T>0$, if there exists $x\in S$ such that $\left\vert u\left(x,t\right)\right\vert \rightarrow +\infty$ as $t\nearrow T^{-}$, or equivalently, $\sum_{x\in S}\left\vert u\left(x,t\right)\right\vert \rightarrow +\infty$ as $t\nearrow T^{-}$.
\end{definition}

In order to state and prove our result, we introduce the following condition:
\[
\hbox{$(C_{p})$ $\hspace{1cm} \alpha F\left(u\right) \leq uf(u)+\beta u^{p}+\gamma,\,\,u>0$}
\]
for some $\alpha>2$, $\beta\geq 0$, and $\gamma>0$ with $0\leq\beta\leq\frac{\left(\alpha-p\right)\lambda_{p,0}}{p}$.

\begin{Rmk}
Observing that $\lambda_{p,0}=0$ if and only if $\sigma\equiv 0$, we can easily obtain that the condition of $\alpha$ in $(C_{p})$ is difference in each $p>1$ and boundary conditions as follows:
\begin{itemize}
	\item[(i)] For all $p>1$, if $\sigma\equiv 0$, then $\alpha>2$.
	\item[(ii)] For all $1<p\leq 2$, if $\sigma\not\equiv 0$, then $\alpha>2$.
	\item[(iii)] For all $p>2$, if $\sigma\not\equiv 0$, then $\alpha\geq p$.
\end{itemize}
\end{Rmk}

We now state the main theorem of this paper:

\begin{Thm}\label{cptheorem}
For $p>1$ and the function $f$ with the hypothesis $(C_{p})$, if the initial data $u_{0}$ satisfies
\begin{equation}\label{11}
-\frac{1}{2p}\sum_{x,y\in\overline{S}}\left\vert u_{0}\left(x\right)-u_{0}\left(y\right)\right\vert^{p}\omega\left(x,y\right) -\frac{1}{p}\sum_{z\in\Gamma}\frac{\sigma(z)}{\mu(z)}|u_{0}(z)|^{p} +\sum_{x\in S}\left[F(u_{0}\left(x\right))-\gamma\right]>0,
\end{equation}
then the solutions $u$ to the equation \eqref{equation} blow up at finite time $T^{*}$ in a sense of
\[
\lim_{t\rightarrow T^{*}}\sum_{x\in S}u^{2}\left(x,t\right)=+\infty,
\]
where $\gamma$ is the constant in the condition $(C_{p})$. \end{Thm}
\begin{proof}
First of all, let us define functionals by
\begin{equation*}
	A(t):=\sum_{x\in S}u^{2}(x,t),\,\,\,t \geq 0
\end{equation*}
and
\begin{equation*}
\begin{aligned}
	B(t):=&-\frac{1}{2p}\sum_{x,y\in\overline{S}}\left\vert u\left(x,t\right)-u\left(y,t\right)\right\vert ^{p}\omega\left(x,y\right)-\frac{1}{p}\sum_{z\in\Gamma}\frac{\sigma(z)}{\mu(z)}|u(z,t)|^{p}\\&+\sum_{x\in S}\left[F\left(u\left(x,t\right)\right)-\gamma\right],\,\,\,t\geq 0.
\end{aligned}
\end{equation*}
Then we have from the equation \eqref{equation} and Lemma \ref{laplace} that
\begin{equation}\label{main11}
\begin{aligned}
	A'(t)=&2\sum_{x\in S}u(x,t)\left[\Delta_{p,\omega}u(x,t)+f(u(x,t))\right]\\=&2\sum_{x\in\overline{S}}u(x,t)\Delta_{p,\omega}u(x,t) +2\sum_{z\in\Gamma}u(z,t)\frac{\partial u}{\partial_{p}n}(z,t)+2\sum_{x\in S}u(x,t)f(u(x,t)) \\=&-\sum_{x,y\in\overline{S}}|u(x,t)-u(y,t)|^{p}\omega(x,y)-2\sum_{z\in\Gamma}\frac{\sigma(z)}{\mu(z)}|u(z,t)|^{p}\\&+2\sum_{x\in\overline{S}}u(x,t)f(u(x,t)).
\end{aligned}
\end{equation}
Applying the condition $(C_{p})$ and Lemma \ref{eigenvalue}, we can see that \eqref{main11} implies
\begin{equation}\label{main22}
\begin{aligned}
	A'(t)\geq& 2\sum_{x\in S}\left[\alpha F(u(x,t))-\beta u^{p}(x,t)-\gamma  \right]-\sum_{x,y\in\overline{S}}|u(x,t)-u(y,t)|^{p}\omega(x,y)\\&-2\sum_{z\in\Gamma}\frac{\sigma(z)}{\mu(z)}|u(z,t)|^{p} \\\geq & 2\alpha B(t)-2\beta \sum_{x\in S}u^{p}(x,t)+ \left(\frac{\alpha}{p}-1\right)\sum_{x,y\in\overline{S}}|u(x,t)-u(y,t)|^{p}\omega(x,y)\\&+2\left(\frac{\alpha}{p}-1\right) \sum_{z\in\Gamma}\frac{\sigma(z)}{\mu(z)}|u(z,t)|^{p}\\ \geq & 2\alpha B(t)+2\left[\frac{(\alpha-p)\lambda_{p,0}}{p}-\beta\right]\sum_{x\in S}u^{p}(x,t)\\ \geq & 2\alpha B(t).
\end{aligned}
\end{equation}
Here, it is easy to see that if $\lambda_{p,0}=0$ or $\alpha = p$, then $\beta=0$. Therefore, even though $\lambda_{p,0}=0$ or $\alpha = p$, \eqref{main22} is true.

On the other hand, we have from the equation \eqref{equation} and Lemma \ref{laplace} that
\begin{equation}\label{main33}
\begin{aligned}
	B'(t)=&-\frac{1}{2}\sum_{x,y\in\overline{S}}|u(y,t)-u(x,t)|^{p-2}[u(y,t)-u(x,t)][u_{t}(y,t)-u_{t}(x,t)]\omega(x,y)\\&-\sum_{z\in\Gamma}\frac{\sigma(z)}{\mu(z)}|u(z,t)|^{p-2}u(z,t)u_{t}(z,t)+\sum_{x\in S}f(u(x,t))u_{t}(x,t)\\=& \sum_{x\in\overline{S}}\Delta_{p,\omega}u(x,t)u_{t}(x,t)+\sum_{z\in\partial S}\frac{\partial u}{\partial_{p}n}(z,t)u_{t}(z,t)+\sum_{x\in S}f(u(x,t))u_{t}(x,t) \\=&\sum_{x\in S}u_{t}(x,t)\left[\Delta_{p,\omega}u(x,t)+f(u(x,t))\right]\\=&\sum_{x\in S}u_{t}^{2}(x,t)\geq 0.
\end{aligned}
\end{equation}
Now, we will show that
\begin{equation}\label{main44}
	\frac{d}{dt}\left[A^{-\frac{\alpha}{2}}(t)B(t)\right]=-\frac{\alpha}{2}A^{-\frac{\alpha}{2}-1}A'(t)B(t)+A^{-\frac{\alpha}{2}}B'(t)\geq 0
\end{equation}
for all $t>0$. Using the Schwarz inequality, we obtain from \eqref{main22} and \eqref{main33} that
\begin{equation*}
\begin{aligned}
	\frac{\alpha}{2} A'(t)B(t)\leq&\frac{1}{4}\left[A'(t)\right]^{2}=\left[\sum_{x\in S}u(x,t)u_{t}(x,t)\right]^{2}\leq \sum_{x\in S}u^{2}(x,t)\sum_{x\in S}u_{t}^{2}(x,t)\\=& A(t)B'(t)
\end{aligned}
\end{equation*}
for all $t>0$. Therefore, the inequality \eqref{main44} is true, which implies that
\begin{equation}\label{main55}
	\frac{1}{2\alpha}A^{-\frac{\alpha}{2}}(t)A'(t)\geq A^{-\frac{\alpha}{2}}(t)B(t)\geq A^{-\frac{\alpha}{2}}(0)B(0)>0.
\end{equation}
Solving the differential inequality \eqref{main55}, we obtain
\begin{equation*}
	A(t)\geq \left[ \frac{1}{ (\alpha-2)\alpha A^{-\frac{\alpha}{2}}(0)B(0)t+A^{\frac{2-\alpha}{2}  }(0)} \right]^{\frac{\alpha-2}{2}}.
\end{equation*}
Hence, $A(t)$ blows up in finite time $T$ with $0<T\leq \frac{A(0)}{(\alpha-2)\alpha B(0)}$.

\end{proof}

\begin{Rmk}
The above blow-up time can be estimated roughly as
\begin{equation*}
	0<T\leq \frac{\frac{1}{(\alpha-2)\alpha}\sum_{x\in S}u_{0}^{2}(x)}{ -\frac{\sum_{x,y\in\overline{S}}|u_{0}(x)-u_{0}(y)|^{p}\omega(x,y)}{2p}-\frac{\sum_{z\in\Gamma}\frac{\sigma(z)}{\mu(z)}u_{0}^{p}(z)}{p}+\sum_{x\in S}\left[F\left(u_{0}(x)\right)-\gamma\right]  }.
\end{equation*}
\end{Rmk}
\begin{Rmk}
	Chung and Choi \cite{CC} obtained the blow-up results for the equation \eqref{equation} under the Dirichlet boundary condition in the continuous setting, where $p\geq 2$ by using the $(C_{p})$ condition. In fact, their condition had assumption $\alpha>p$, which is one of main difference to us.
\end{Rmk}

\section{Discussion on the Condition $(C_{p})$ with the initial data conditions}

As seen in the proof of Theorem \ref{cptheorem}, the concavity method is a tool for deriving the blow-up solution via the auxiliary function $B(t)$ under the condition $(A_{p})$, $(B_{p})$, or $(C_{p})$, by imposing $B(0)>0$, instead of the large initial data. In this section, we compare the conditions $(A_{p})$, $(B_{p})$, and $(C_{p})$ each other and discuss the role of $B(0)>0$.

First of all, we consider the Neumann boundary condition ($\sigma\equiv 0$). Summing up over $S$ to the equation \eqref{equation}, we have
\begin{equation*}
\begin{aligned}
	\sum_{x\in S}u_{t}(x,t)=&\sum_{x\in\overline{S}}\Delta_{p,\omega}u(x,t)-\sum_{z\in\partial S}\Delta_{p,\omega}u(z,t)+\sum_{x\in S}f(u(x,t))\\=& \sum_{x\in S}f(u(x,t)).
\end{aligned}
\end{equation*}
From the above equality, we can obtain that the time-behavior of $\sum_{x\in S}u(x,t)$ is determined by $\sum_{x\in S}f(u(x,t))$. Therefore, by the definition of the blow-up, we can expect that the blow-up condition for the solution $u$ depends only on $f$, not on $p$. On the other hand, for all $p>1$, the $(C_{p})$ condition is represented by
$$
(2+\epsilon)F(u) \leq uf(u)+\gamma
$$
for some $\epsilon>0$ and $\gamma>0$, which also doesn't depend on $p$.

From now on, we consider the boundary condition $\sigma\not\equiv 0$. Let us recall the conditions as follows:\\
for $1<p\leq 2$ : 
\[
\begin{aligned}
&\hbox{$(A_{p})$ $\hspace{3mm} (2+\epsilon) F(u)\leq uf(u)$},\\
&\hbox{$(B_{p})$ $\hspace{3mm} (2+\epsilon) F(u)\leq uf(u) + \gamma$},\\
&\hbox{$(C_{p})$ $\hspace{3mm} (2+\epsilon) F\left(u\right) \leq uf(u)+\beta u^{p}+\gamma$},
\end{aligned}
\]
where
$$
\epsilon>0,\,\,0\leq\beta\leq\frac{(2+\epsilon-p)\lambda_{p,0}}{p},\,\,\text{ and } \,\,\gamma>0,
$$
and for $p> 2$ : 
\[
\begin{aligned}
&\hbox{$(A_{p})$ $\hspace{3mm} (p+\epsilon) F(u)\leq uf(u)$},\\
&\hbox{$(B_{p})$ $\hspace{3mm} (p+\epsilon) F(u)\leq uf(u) + \gamma$},\\
&\hbox{$(C_{p})$ $\hspace{3mm} (p+\epsilon) F\left(u\right) \leq uf(u)+\beta u^{p}+\gamma$},\\
\end{aligned}
\]
where
$$
\epsilon\geq 0,\,\,0\leq\beta\leq\frac{\epsilon\lambda_{p,0}}{p},\,\,\text{ and } \,\,\gamma>0
$$
for every $u\geq0$. Here, $F(u):=\int_{0}^{u}f(s)ds$. 

It is easy to see that $(A_{p})$ implies $(B_{p})$, $(B_{p})$ implies $(B_{p}')$, and $(B_{p})$ implies $(C_{p})$, in turn. In fact, the conditions $(A_{p})$, $(B_{p})$, and $(B_{p}')$ are independent of the first eigenvalue $\lambda_{p,0}$ which depends on the domain. However, the condition $(C_{p})$ depends on the domain, due to the term $\beta u^{p}$. From this point of view, the condition $(C_{p})$ can be understood as a refinement of $(B_{p})$, corresponding to the domain. On the contrary, if a function $f$ satisfies $(C_{p})$ for every domain $\overline{S}$, then the first eigenvalue $\lambda_{p,0}$ can be arbitrary small so that the condition $(C_{p})$ get closer to $(B_{p})$ arbitrarily. Besides, as far as the authors know, there has not been any noteworthy condition for the concavity method other than $(A_p)$ or $(B_{p})$.

\begin{Rmk}
	In fact, there has been many efforts to obtain a condition $\epsilon=0$ in the continuous analogue. For example, Junning was studied the blow-up solutions to the equation \eqref{equation} in the continuous setting under the Dirichlet boundary condition with the assumption $\epsilon=0$ in $(C_{p})$ and the initial data $u_{0}$ satisfying
	\begin{equation*}
		-\frac{1}{p}\int_{\Omega}|\nabla u_{0}(x)|^{p}dx +\int_{\Omega} F(u_{0}(x))dx\geq \frac{4(p-1)}{T(p-2)^{2}p}\int_{\Omega}u_{0}^{2}(x)dx,
	\end{equation*}
	where $p>2$ and $\Omega\subset \mathbb{R}^{N}$(see \cite{J}). From this point of view, for $p>2$, our condition $\epsilon\geq 0$ with $B(0)>0$, which is one of our meaningful result, refines the conventional results.
\end{Rmk}

Now we will consider the case $p>2$ and $1<p\leq 2$ to investigate the conditions $(A_{p})$, $(B_{p})$, and $(C_{p})$.

Case 1: $p>2$.\\
Assuming $\epsilon>0$ we obtain that the condition $(C_{p})$ is equivalent to
\begin{equation}\label{equi1}
\frac{d}{du}\left(\frac{F(u)}{u^{p+\epsilon}}-\frac{\gamma}{p+\epsilon}\cdot\frac{1}{u^{p+\epsilon}}-\frac{\beta}{\epsilon}\cdot\frac{1}{u^{\epsilon}}\right)\geq0,\,\,u>0.
\end{equation}
By the similar way, assuming $\epsilon=0$ we have 
\begin{equation}\label{equi2}
\frac{d}{du}\left(\frac{F(u)}{u^{p}}-\frac{\gamma}{p}\cdot\frac{1}{u^{p}}\right)\geq0,\,\,u>0.
\end{equation}
Hence, \eqref{equi1} and \eqref{equi2} imply that for every $u>0$ and $p>2$,
\begin{equation*}
\begin{aligned}
&\hbox{$(A_{p})$ holds if and only if $F(u)={u^{p+\epsilon}}h_{1}(u)$},\\
&\hbox{$(B_{p})$ holds if and only if $F(u)={u^{p+\epsilon}}h_{2}(u)+b$},\\
&\hbox{$(C_{p})$ holds if and only if $F(u)={u^{p+\epsilon}}h_{3}(u)+au^{p}+b$},
\end{aligned}
\end{equation*}
for some constants $\epsilon\geq 0$, $a\geq 0$, and $b>0$ with $0\leq a\leq\frac{\lambda_{p,0}}{p}$, where $h_{1}$, $h_{2}$, and $h_{3}$ are nondecreasing function on $(0,+\infty)$. Here also, the constants $\epsilon, a, \hbox{ and } b$ may be different in each case. We note here that the nondecreasing function $h_{1}$ is nonnegative on $(0,+\infty)$, but $h_{2}$, and $h_{3}$ may not be nonnegative, in general.\\

Case 2: $1<p\leq 2$.\\
We obtain that $(C_{p})$ is equivalent to
\begin{equation*}
\frac{d}{du}\left(\frac{F(u)}{u^{2+\epsilon}}-\frac{\gamma}{2+\epsilon}\cdot\frac{1}{u^{2+\epsilon}}-\frac{\beta}{2+\epsilon-p}\cdot\frac{1}{u^{2+\epsilon-p}}\right)\geq0,\,\,u>0,
\end{equation*}
which implies that for every $u>0$ and $1<p\leq 2$,

\begin{equation*}
\begin{aligned}
&\hbox{$(A_{p})$ holds if and only if $F(u)={u^{2+\epsilon}}h_{1}(u)$},\\
&\hbox{$(B_{p})$ holds if and only if $F(u)={u^{2+\epsilon}}h_{2}(u)+b$},\\
&\hbox{$(C_{p})$ holds if and only if $F(u)={u^{2+\epsilon}}h_{3}(u)+au^{p}+b$},
\end{aligned}
\end{equation*}
for some constants $\epsilon>0$, $a\geq 0$, and $b>0$ with $0\leq a\leq\frac{\lambda_{p,0}}{p}$, where $h_{1}$, $h_{2}$, and $h_{3}$ are nondecreasing function on $(0,+\infty)$. Here, $h_{2}$, and $h_{3}$ may not be nonnegative, in general.\\

\begin{Rmk}
	Chung and Choi studied the case $f(u)=\lambda u^{q}$ in the Dirichlet boundary condition with respect to blow-up property (see \cite{CC,C2}). In their results, the solution $u$ blows up in finite time if
	\begin{itemize}
		\item[(i)] $0<p-1<q$, $q>1$, and the initial data $u_{0}$ is sufficiently large.
		\item[(ii)] $1<p-1=q$ and $\lambda>\lambda_{p,0}$.
	\end{itemize}
	Considering the case (i) and (ii), we obtain that the solution $u$ doesn't blow up in finite time whenever $q\leq 1$. From this observation, we can easily obtain that $\epsilon$ in the $(C_{p})$ condition cannot be $0$ when $1<p\leq 2$, since $$F(u)={u^{2+\epsilon}}h_{3}(u)+au^{p}+b.$$
\end{Rmk}

\begin{Thm}\label{C condition f>=a u^1+e}
	For $p>1$, let $f$ be a real-valued function satisfying the condition $(C_{p})$. Suppose that $f(u)\geq \lambda u^{p-1}$, $u>0$ for some $\lambda>\lambda_{p,0}$. Then the following statements are true.
	\begin{itemize}
		\item[(i)] There exists $m>0$ such that $h_{3}(u)>0$ for $u\geq m$.
		\item[(ii)] There exists $\zeta>0$ such that $f(u)\geq \zeta u^{\max\{p-1, 1\}+\epsilon}$, $u\geq m$.
		\item[(iii)] The conditions $(B_{p})$ and $(C_{p})$ are equivalent when $p\geq 2$.
	\end{itemize}
\end{Thm}

\begin{proof}
	$(i)$: First, it follows from the fact $F(u)\geq \frac{\lambda}{p}u^{p}> \frac{\lambda_{p,0}}{p}u^{p}$ that
	\[
	u^{\max\{p, 2\}+\epsilon}h_{3}(u)=F(u)-au^{p}-b\geq \frac{\lambda-\lambda_{p,0}}{p}u^{p}-b,
	\]	
	which goes to $+\infty$, as $u\rightarrow+\infty$. Therefore, we can find $m>0$ such that $h_{3}(m)>0$.\\
	$(ii)$: $(i)$ implies that
	\[
	F(u)\geq u^{\max\{p, 2\}+\epsilon}h_{3}(u),\,\, u\geq m.
	\]
	Putting it into the condition $(C_{p})$, we obtain
	\[
	\alpha u^{\max\{p, 2\}+\epsilon}h_{3}(m)\leq uf(u) + \beta u^{p} + \gamma.
	\]
	Hence, we obtain that
	\[
	\alpha u^{\max\{p-1, 1\}+\epsilon}h_{3}(m)\leq f(u)  + \beta u^{p-1}+\frac{\gamma}{u}\leq \left(1 + \frac{\beta}{\lambda_{p,0}}\right)f(u) + \gamma,\,\,u\geq m>0,
	\]
	which gives
	\[
	f(u)\geq \zeta u^{1+\epsilon},\,\,u\geq m>0
	\]
	for some $\zeta>0$.\\
	Now consider the case $p\geq 2$. Since $0\leq \beta\leq\frac{\epsilon\lambda_{p,0}}{p}$ and $f(u)\geq \lambda u>\lambda_{p,0}u$, $u>0$, it follows from $(C_{p})$ that
	\[
	\epsilon_{1} F\left(u\right) + \left(p+\epsilon_{2}\right) F\left(u\right) \leq uf(u)+\frac{\epsilon\lambda_{p,0}}{p} u^{p}+\gamma,
	\]
	where $\epsilon_{1}= \frac{\epsilon\lambda_{p,0}}{\lambda}\geq 0$ and $\epsilon_{2}=\epsilon-\epsilon_{1}\geq0$. This implies that for every $u>0$,
	\[
	\begin{aligned}
	uf(u)+\gamma &\geq \left(p+\epsilon_{2}\right) F\left(u\right) + \epsilon_{1} \int_{0}^{u}\left[f\left(s\right)-\lambda s\right]ds\\
	&\geq \left(p+\epsilon_{2}\right) F\left(u\right),
	\end{aligned}
	\]
	which implies $(B_{p})$.
%	Since $0<\beta\leq\frac{(2+\epsilon-p)\lambda_{p,0}}{p}$ and $f(u)\geq \lambda u^{p-1}>\lambda_{p,0}u^{p-1}$, $u>0$, it follows from $(C_{p})$ that
%	\[
%	\epsilon_{1} F\left(u\right) + \left(p+\epsilon_{2}\right) F\left(u\right) \leq uf(u)+\frac{(2+\epsilon-p)\lambda_{p,0}}{p} u^{p}+\gamma,
%	\]
%	where \rc{$\epsilon_{1}=\frac{(2+\epsilon-p)\lambda_{p,0}}{\lambda}>0$} and $\epsilon_{2}=\epsilon-\epsilon_{1}>0$. This implies that for every $u>0$,
%	\[
%	\begin{aligned}
%	uf(u)+\gamma &\geq \left(p+\epsilon_{2}\right) F\left(u\right) + \epsilon_{1} \int_{0}^{u}\left[f\left(s\right)-\lambda s^{p-1}\right]ds\\
%	&\geq \left(p+\epsilon_{2}\right) F\left(u\right),
%	\end{aligned}
%	\]
%	which gives $(B_{p})$.
\end{proof}

In general, only the condition $(C_{p})$ may not guarantee the blow-up solutions for every initial data $u_{0}$. Therefore, from now on, we are going to discuss when we can find initial data $u_{0}$ satisfies $B(0)>0$.
\begin{Lemma}
	Let $p>1$. If there exists $v_{0}>0$ such that $F(v_{0})>\frac{\omega_{0}}{p}v_{0}^{p}+\gamma_{1}$, where $\gamma_{1}\geq\gamma$, then there exists the initial data $u_{0}$ such that $B(0)>0$. Here, $\omega_{0}:=\max_{x\in S}d_{\omega}x$.
\end{Lemma}

\begin{proof}
	First of all, there exist $a,b>0$ with $0<a<b$ such that $F(v)>\frac{\omega_{0}}{p}v^{p}+\gamma_{1}$, $v\in(a,b)$, since $F$ is continuous on $[0,+\infty)$. Now, we consider the function $u_{0}(x)$ satisfying
	\[
	\begin{cases}
	a<u_{0}\left(x\right)<b,\,\,\, & x\in S,\\
	0<u_{0}(x)<b,\,\,\,&x\in \Gamma,\\
	u_{0}\left(x\right)=0,&x\in \partial S\setminus \Gamma,
	\end{cases}
	\]
	which satisfies the boundary condition $B[u_{0}]=0$. Then we obtain that
	\[
	\begin{aligned}
	B\left(0\right) &=\frac{1}{p}\sum_{x\in S}\sum_{y\in\overline{S}}|u_{0}(y)-u_{0}(x)|^{p-2}\left[u_{0}(y)-u_{0}(x)\right]u_{0}(x)\omega(x,y)\\&+\sum_{x\in S}\left[F(u_{0}\left(x\right))-\gamma\right]\\
	& \geq  -\frac{1}{p}\sum_{x\in S}\sum_{y\in\overline{S}}u_{0}^{p}\left(x\right)\omega\left(x,y\right)+\sum_{x\in S}\left[F(u_{0}\left(x\right))-\gamma\right]\\
	& = -\frac{1}{p}\sum_{x\in S}u_{0}^{p}\left(x\right)d_{\omega}x+\sum_{x\in S}[F(u_{0}\left(x\right))-\gamma]\\
	& \geq \sum_{x\in S}\left[F(u_{0}\left(x\right))-\frac{\omega_{0}}{p}u_{0}^{p}\left(x\right)\right]-\gamma|S|\\
	&>\gamma_{1}|S|-\gamma|S|\geq0,
	\end{aligned}
	\]
	where $|S|$ denotes the number of vertices in $S$.
\end{proof}
\begin{Cor} The following statements are true.
	\begin{enumerate}[(i)]
		\item 	If there exists $(a,b)$ such that $F(v)>\frac{\omega_{0}}{p}v^{p}+\gamma_{1}$, $\gamma_{1}\geq \gamma$ for every $v\in(a,b)$, then for every $u_{0}$ satisfying the boundary condition $B[u_{0}]=0$ such that
		\[
		\begin{cases}
		a<u_{0}\left(x\right)<b,\,\,\, & x\in S,\\
		0<u_{0}(x)<b,\,\,\,&x\in \Gamma,\\
		u_{0}\left(x\right)=0,&x\in \partial S\setminus \Gamma,
		\end{cases}
		\]
		we see that $B(0)>0$.
		\item If $F(v)>\omega_{0}v^{\max\{2+\epsilon_{1},p\}}+\gamma_{1}$, $\epsilon_{1}>0$, $\gamma_{1}\geq\gamma$, for every $v\in[1,+\infty)$, then the solutions blow up for every initial data $u_{0}>0$. Here, $\omega_{0}:=\max_{x\in S}d_{\omega}x$.
	\end{enumerate}
\end{Cor}
\section*{Conflict of Interests}
\noindent The authors declare that there is no conflict of interests regarding the publication of this paper.

%%%%%%%%%%%%%%%%%%%%%%%%%%%%%%%%%%%%%%%%%%%%%%%%%%%%%%%%%%%%%%%%%%%%%%%%%%%%%%%%%%%%%%%%%%%%%%%%%%%%%%%%%%%%%%%%%%%%%%
\section*{Acknowledgments}
\noindent This work was supported by Basic Science Research Program through the National Research Foundation of Korea(NRF) funded by the Ministry of Education (NRF-2015R1D1A1A01059561).
%%%%%%%%%%%%%%%%%%%%%%%%%%%%%%%%%%%%%%%%%%%%%%%%%%%%%%%%%%%%%%%%%%%%%%%%%%%%%%%%%%%%%%%%%%%%%%%%%%%%%%%%%%%%%%%%%%%%%%%%%%%%%%%%%%%%%%%%%%%%%%%%%%%%%%
%%%%%%%%%%%%%%%%%%%%%%%%%%%%%%%%%%%%%%%%%%%%%%%%%%%%%%%%%%%%%%%%%%%%%%%%%%%%%%%%%%%%%%%%%%%%%%%%%%%%%%%%%%%%%%%%%%%%%%%%%%%%%%%%%%%%%%%%%%%%%%%%%%%%%%

\end{document}